\documentclass[a4paper,9pt,twoside]{extarticle}

\usepackage{graphicx}
\usepackage{caption}
\usepackage{subcaption}
\usepackage{multirow}
\usepackage[utf8]{inputenc}
\usepackage[T1]{fontenc}
\usepackage{ dsfont }
\usepackage[english]{babel}

\usepackage{charter}

\usepackage{indentfirst}
\usepackage{xcolor}
\usepackage{verbatim}

\usepackage{hyperref}

\usepackage{pifont}

\usepackage[top=3.cm, bottom=4.0cm, left=3.5cm, right=3.5cm]{geometry}
\usepackage{amsmath}
\usepackage{cases}
\usepackage{amsthm}	
\usepackage{amsfonts}	
\usepackage{amssymb	
            ,bbm
            ,units 
            ,stmaryrd
           }
\usepackage{enumerate}

\usepackage[square,numbers,sort&compress]{natbib}

\usepackage{
            ulem		
           ,soul		
} 

\numberwithin{equation}{section}
\numberwithin{figure}{section}
 \usepackage[nodayofweek]{datetime}

\renewcommand*{\thefootnote}{\fnsymbol{footnote}}

\title{Malliavin differentiability of McKean-Vlasov SDEs with common noise}

\author{
Jianhai Bao\footnote{Center for Applied Mathematics, Tianjin University, 300072 Tianjin, P.R.~China.
{\tt jianhaibao@tju.edu.cn}
}
\and
  Gon{\c c}alo~dos~Reis\footnote{University of Edinburgh, School of Mathematics, Edinburgh, EH9 3FD, UK and 
Center for Mathematics and Applications (NOVA Math),  UNL, PT.  
{\tt G.dosReis@ed.ac.uk}
    }\, 
    \footnote{G.~dos Reis acknowledges support from the FCT – Fundação para a Ciência e a Tecnologia, I.P., under the scope of the projects UIDB/00297/2020 (https://doi.org/10.54499/UIDB/00297/2020) and UIDP/00297/2020 (https://doi.org/10.54499/UIDP/00297/2020) (Center for Mathematics and Applications, NOVA Math), and by the UK Research and Innovation (UKRI) under the UK government’s Horizon Europe funding Guarantee [Project UKRI343].}
  \and 
  Zac Wilde
  \footnote{University of Edinburgh, School of Mathematics, Edinburgh, EH9 3FD, UK. 
{\tt zwilde@ed.ac.uk}  
           }
        }
        
\date{%
    \footnotesize 
    \longdate \today 
}

\theoremstyle{plain}
\newtheorem{theorem}{Theorem}[section]
\newtheorem{lemma}[theorem]{Lemma}

\newtheorem{definition}[theorem]{Definition}

\newtheorem{remark}[theorem]{Remark}

\newcommand{\bD}{\mathbb{D}}
\newcommand{\bE}{\mathbb{E}}

\newcommand{\bN}{\mathbb{N}}
\newcommand{\bP}{\mathbb{P}}

\newcommand{\bR}{\mathbb{R}}

\newcommand{\cC}{\mathcal{C}}

\newcommand{\cF}{\mathcal{F}}

\newcommand{\cP}{\mathcal{P}}

\newcommand{\trace}{\textrm{Trace}}
\newcommand{\Supp}{\textrm{Supp}}
\newcommand{\Law}{\textrm{Law}}

\newcommand{\dd}{\mathrm{d}}

\begin{document}

\selectlanguage{english}

\maketitle
\renewcommand*{\thefootnote}{\arabic{footnote}}

\begin{abstract} 
We establish the Malliavin differentiability of McKean-Vlasov stochastic differential equations (MV-SDEs) with common noise under the global  Lipschitz assumption in the space variable and the measure variable. 
Our result gives also meaning to the Malliavin derivative of the conditional law with respect to the common noise. 
As an application, we derive an integration by parts formula on the Wiener space for the class of common noise MV-SDEs under consideration.  
\end{abstract}





%
%
%

\section{Introduction}

McKean-Vlasov stochastic differential equations (SDEs for short) are a class of SDEs whose coefficients have dependence not only on the state process but also on 
the law of the solution 
process \cite{book:CarmonaDelarue2018a}. 
In the present work, we turn to  focus on McKean-Vlasov SDEs with common noise, which, in literature,  are also termed  as the conditional McKean-Vlasov SDEs \cite{LSZ}. As far as such SDEs are concerned, there are two sources of noise input into the system, in which one is  an idiosyncratic source of randomness and the other is a so-called common noise (i.e., the path $(\omega^0_t)_{t\ge0}$ driven by the Brownian motion which we condition upon pathwise). More precisely, in this paper we    work on  the following dynamics on $\mathbb R^d$: 
\begin{align}
\label{Eq:General MVSDE}
\dd X_t = b(t,X_t,\mu_t)\dd t +\sigma^0(t,X_t,\mu_t) \dd W^{0}_t
+\sigma^1(t,X_t,\mu_t) \dd W^{1}_t, \quad t>0; 
\quad X_{0} =\xi,
\end{align}
where 
 $b:[0,\infty) \times \bR^d \times\cP(\bR^d) \to \bR^d$, $\sigma^0 :[0,\infty) \times \bR^d \times \cP(\bR^d) \to \bR^{d\otimes m_0}$,  and $\sigma^1 :[0,\infty) \times \bR^d \times \cP_2(\bR^d) \to \bR^{d\otimes m}$, 
the Brownian motion $(W_t^1)_{t\ge0}$, supported on $(\Omega^1,\mathcal F^1,(\mathcal F_t^1)_{t\ge0}, \mathbb P^1)$,  is referred to as an $m$-dimensional idiosyncratic noise, the Brownian motion $(W_t^0)_{t\ge0}$, carried on $(\Omega^0,\mathcal F^0,(\mathcal F_t^0)_{t\ge0}, \mathbb P^0)$,
as an $m_0$-dimensional common noise, 
and $\mu_t:= \mu_t(\omega^0)= \Law(X_t(\omega^0,\cdot))$ for $\bP^0$-almost   $\omega^0$ (i.e., $\mu_t=\bP^1\circ (\bE\big[X_t|\cF^0_t \big])^{-1}$).
Hereinbefore, $\cP(\bR^d)$ is the space of probability measures over $\bR^d$, and  
$\mu_t $ can be understood as a random variable mapping $(\Omega^0,\cF^0,\bP^0)$ to $\cP(\bR^d)$; see \cite[Section 4.3]{book:CarmonaDelarue2018b} or \cite{platonovdosreis2023ItowentzelLions} for related  details. 

As mentioned previously, the drift and diffusion coefficients in \eqref{Eq:General MVSDE} depend on the path $(\omega^0_t)_{t\ge0}$. Thus, one can interpret \eqref{Eq:General MVSDE} in the light of the theory on rough differential equations (RDEs for short).  
One of our contributions lies in extending \cite{bugini2024malliavin}, which addressed    the Malliavin differentiability of standard RDEs,   to the mean-field setting, where   the rough path under consideration  is a Brownian motion path. 
According to \cite[p.110-112]{book:CarmonaDelarue2018b}, the random distribution flow $(\mu_t)_{t>0}$ corresponding to  \eqref{Eq:General MVSDE} solves the following SPDE:
\begin{equation*} 
\dd\mu_t= \Big(-\mbox{div}\big( b(\cdot,\mu_t)\mu_t \big)
+\frac{1}{2}\mbox{trace}\big(\nabla^2\big((\sigma^1(\sigma^1)^\top)(\cdot,\mu_t)\mu_t\big)\big)\Big)\,\dd t
-\mbox{div}\big(\big(\sigma^0(\cdot,\mu_t)\dd W_t^0\big)\mu_t\big),
\end{equation*}
where $(\sigma^1)^\top$ denotes the transpose of $\sigma^1$, $\nabla^2$ means the second-order gradient operator, and $\mbox{div}$
represents the divergence operator. The aforementioned  SPDE, which is also named as the nonlinear stochastic Fokker-Planck equation, is understood  in the weak sense. 
 It is worthy to stress that the flow $(\mu_t)$ involved in \eqref{Eq:General MVSDE}  is random whereas the same counterpart associated with the standard McKean-Vlasov SDEs (without common noise) is deterministic. 
This discrepancy introduces involved  difficulties into the analysis of the Malliavin differentiability for $(X_t)_{t\ge0}$ governed  by  \eqref{Eq:General MVSDE}.

In the past few years, the Malliavin differentiability of classical SDEs advanced greatly. In particular, via a Picard approximation approach, \cite[Theorem 2.2.1]{nualart2006malliavin}
handled the case that  the drift term and the diffusion term involved are globally Lipschitz.  Nevertheless,  the Malliavin 
differentiability for SDEs with drifts of super-linear growth is left open for a long time, which is extremely unexpected. In \cite{imkeller2018differentiability}, the issue mentioned above on the Malliavin 
 differentiability is addressed by the aid of stochastic Gâteaux differentiability and ray absolute continuity. Recently,  the research on the Malliavin 
 differentiability for McKean-Vlasov SDEs 
has also made some progress; see, for instance, 
\cite[Proposition 3.5]{RW} for the globally Lipschitz setting. Since the distribution  variable involved in the coefficients associated with McKean-Vlasov SDEs is deterministic,  
there is no additional term in the linear SDE solved by the Malliavin derivative; see, e.g., \cite[(3.26)]{RW} or \cite[(2)]{reis2023malliavin}. When the coefficients of McKean-Vlasov SDEs under consideration is Lipschitz continuous in the measure variable  under the $L^2$-Wasserstein distance, 
the Malliavin differentiability was  established in \cite[Theorem 3.5]{reis2023malliavin} and \cite[Proposition 4.8]{reis2023malliavin} via the Picard iteration trick and the interacting
particle system (IPS for short) limits
when the underlying drifts satisfy the one-sided Lipschitz condition and fulfil the globally Lipschitz condition in the spatial variable, respectively.

\medskip

\emph{Our contribution: Malliavin differentiability of  McKean-Vlasov  SDEs with common noise under the global Lipschitz assumption.} Specially, 
we extend the classical Malliavin variational results to   McKean-Vlasov SDEs with common noise, where the deterministic coefficients satisfy  a global space-measure Lipschitz condition. To derive the Malliavin differentiability with respect to the common noise, we adopt a similar approach to that of \cite{reis2023malliavin} by appealing to the celebrated \cite[Lemma 1.2.3]{nualart2006malliavin}. In detail, we employ the following strategy: we first 
 study the Malliavin differentiability of the underlying  IPS  followed by passing to the limit via Propagation of Chaos (PoC for short), and then studying how the Malliavin regularity transfers across the particle limit to the limiting equation. Essentially from \eqref{Eq:General MVSDE}, our construction allows one to give meaning to the Malliavin derivative of the conditional distribution flow $\mu_t $ in relation to $W^0$;  see Theorem \ref{MD0} below.

With regarding to the McKean-Vlasov SDEs, the Bismut-Elworthy-Li formulae were provided   in \cite{Crisan2018} for the associated decoupled SDEs with deterministic initial value. Later on,  the Bismut formula for Lions derivatives  was initially established in \cite{RW} for distribution dependent SDEs, which the initial value is a random variable. Subsequently, \cite{RW} was extended to distribution-path dependent SDEs, where the associated drift term  satisfies a monotone condition in the spatial variable. In the present paper, as an application of our main theory derived,  we seek  to establish a conditional integration by parts formula (instead of the corresponding Bismut-Elworthy-Li formula, where the order between the Markov operator and the $L$-derivative is different).

\section{Notation and preliminary results}
\subsection{Notation and spaces}  
For $M,N\in\bR^{m\otimes n}$,   define the inner product $\langle M,N\rangle_{{\rm HS}} = \trace(M^\top N)$, which induces the associated Hilbert-Schmidt norm $\|\cdot\|_{\rm HS}$.  
For $g:\bR^m\to\bR^n$, write $g(x) = (g^1(x),\ldots,g^n(x))$, and define $\nabla_x g$ as the Jacobian matrix 
$(\partial_ig^j)_{i,j}$ with $\partial_i:=\frac{\partial}{\partial x^i}$ for $1\leq i \leq m$, $1\leq j\leq n.$  
For $t>0$ and  a  vector-valued or matrix-valued function $f$ defined on $\bR^d$, we set the uniform $\|f\|_{0,t}:=\sup_{0\le s\le t}\|f(s)\|_{\rm HS}$. For a given terminal $T>0,$   write  $\mathcal{C}_T:=C([0,T];\mathbb R^d)$, which is endowed with the uniform norm $\|\cdot\|_{\mathcal C_T}$, as the path space of continuous functions $\omega:[0,T]\rightarrow \mathbb R^d$. 
Throughout the paper, we shall work on the stochastic basis $(\Omega, \mathcal F, (\mathcal F_t)_{t\ge0},\mathbb P)$, where $\Omega=\Omega^0\times \Omega^1$, $(\mathcal F,\mathbb P)$ is the completion of $(\mathcal F^0\times \mathcal F^1, \mathbb P^0\times\mathbb P^1)$, and $(\mathcal F_t)_{t\ge0}$
is the complete and right continuous augmentation of $(\mathcal F_t^0\times\mathcal F_t^1)_{t\ge0}$.  
Let $\bE^0$, $\bE^1$ and $\bE$ denote the expectation operators under the probability measures  $\bP^0$,  $\bP^1$, and  $\bP$, respectively.
 The space of probability measures over $\bR^d$ with finite $r$th moment for $r\ge1$, written as  $\cP_r(\bR^d)$,  is a Polish space  under the Wasserstein distance: 
\begin{align*}
\mathcal W_r(\mu,\nu) = \inf_{\pi\in\Pi(\mu,\nu)} \Big(\int_{\bR^d\times \bR^d} |x-y|^r\pi(\dd x,\dd y)\Big)^\frac1r, \quad \mu,\nu\in \cP_r(\bR^d) ,
\end{align*}   
where $\Pi(\mu,\nu)$ (i.e., the family of probability measures on $\mathbb R^d\times\mathbb R^d$) is the set of couplings for $\mu$ and $\nu$ (that is, $\pi\in\Pi(\mu,\nu)$  satisfying  $\pi(\cdot\times \bR^d)=\mu(\cdot)$ and $\pi(\bR^d \times \cdot)=\nu(\cdot)$). 
Let $\Supp(\mu)$ denote the support of the probability measure $\mu$.


\subsection{Basics on Malliavin calculus}
In this subsection, we aim to recall some basics on Malliavin calculus to make the content self-contained; see e.g. the monographs  \cite{nualart2006malliavin} and \cite[pages 8-9]{W2013}. For $T>0$ fixed, denote by $\mathcal H$ the Cameron-Martin space defined as below:
$$ 
\mathcal H=\Big\{h\in\mathcal C_T\Big|h_0={\bf 0}, h_t' \mbox{ exists a.e. in }t, \|h\|_{\mathcal H}^2:=\int_0^T|h_t'|^2\dd t<\infty\Big\},
$$ 
which is a Hilbert space with the inner product $\langle f,g\rangle_{\mathcal H}:=\int_0^T\langle f_t',g_t'\rangle\dd t$.
Let $\mu_T$ be the distribution of $W_{[0,T]}=(W_t)_{t\in[0,T]}$ on the path space $\mathcal C_T$ of paths starting at the origin, where $(W_t)_{t\ge0}$ is a $d$-dimensional Brownian motion. In fact, $\mu_T$
is the so-called Wiener measure so that the coordinate process $W_t(\omega)=\omega_t$ is a $d$-dimensional Brownian motion. For $F\in L^2(\mathcal C_T,\mu_T)$, $F(W_{[0,T]})$ is called Malliavin differentiable along direction $h\in\mathcal H$ if the directional derivative 
$
D_hF(W_{[0,T]}): =\lim_{\varepsilon\rightarrow 0} \frac{1}{\varepsilon}  (F(W_{[0,T]}+\varepsilon h)-F(W_{[0,T]}))
$ 
exists in $L^2(\mathcal C_T,\mu_T)$. If $\mathcal H\ni h\mapsto D_hF\in L^2(\mathcal C_T,\mu_T)$ is bounded, via the Riesz representation theorem,  then there exists a unique $DF(W_{[0,T]})\in L^2(\mathcal C_T\rightarrow \mathcal H,\mu_T)$ such that $\langle DF(W_{[0,T]}),h\rangle =D_hF(W_{[0,T]}) $ holds true  $\mu_T$-a.s. In this case, we write $F(W_{[0,T]})\in\mathcal D(D)$ (the domain of $D$)
and call $DF(W_{[0,T]})$ the Malliavin gradient (derivative) of $F(W_{[0,T]})$. The adjoint operator $(\delta, \mathcal D(\delta))$ of $(D,\mathcal D(D))$ is called the Malliavin divergence. Via the integration by parts formula, we have 
\begin{align*}
  \mathbb E(D_hF)=\mathbb E(F\delta(f)) ,\quad F\in \mathcal D(D),\, h\in \mathcal D(\delta).
\end{align*}
In particular, when $h\in L^2(\mathcal C_T\rightarrow \mathcal H, \mu_T)$ is an adapted stochastic process, then $h\in\mathcal D(\delta)$ and $\delta(h)=\int_0^T\langle h_t',\dd W_t\rangle.$ Otherwise, $\delta(h)$
is the Skorokhod stochastic integral.

\subsection{Lions differentiability and the empirical projection map}
\label{sec:MeasureDerivatives}
To begin, we recall the notion of the Lions derivative; see e.g. \cite[Definition 1.1]{RW}, which indeed coincides with the   one given in \cite{book:CarmonaDelarue2018a} via a lift-up approach.
 The functional   $f:\mathcal P_2(\mathbb R^d)\rightarrow\mathbb R$ is called $L$-differentiable at $\mu\in\mathcal P_2(\mathbb R^d) $ if the functional $L^2(\mathbb R^d\rightarrow\mathbb R^d,\mu)\ni\phi\mapsto f(\mu\circ(I+\phi)^{-1})$ is Fr\'{e}chet differentiable at ${\bf 0}\in L^2(\mathbb R^d\rightarrow\mathbb R^d,\mu)$, where $I$ means the identity map and $\mu\circ(I+\phi)^{-1}$ is the push-forward measure of $\mu$ by  $I+\phi$. 
Namely, there is (hence, unique) $\gamma\in 
L^2(\mathbb R^d\rightarrow\mathbb R^d,\mu) $ such that
\begin{align*}
\lim_{\mu(|\phi|^2)\rightarrow0}\frac{f(\mu\circ(I+\phi)^{-1})-f(\mu)-\mu(\langle\gamma,\phi\rangle)}{\sqrt{\mu(|\phi|^2)}}=0.
\end{align*}
In this case, we write $\partial_\mu f(\mu)=\gamma$ and call it the  $L$-derivative of $f$ at $\mu$. For a vector-valued or matrix-valued function $f$ on $\mathbb R^d$, the $L$-derivative (if it exists) is defined in the component-wise way. 
If the $L$-derivative $\partial_\mu f(\mu)$  exists for all $\mu\in \mathcal P_2(\mathbb R^d)$, 
then $f$ is called $L$-differentiable. If, moreover, for any $\mu\in\mathcal P_2(\mathbb R^d)$  there exists a $\mu$-version $\partial_\mu f(\mu)$ such that $\partial_\mu f(\mu)(x)$ is jointly continuous in $(x,\mu)\in\mathbb R^d\times\mathcal P_2(\mathbb R^d)$, we write $f\in C^{(1,0)}(P_2(\mathbb R^d))$. $g$ is called differentiable on $\mathbb R^d\times\mathcal P_2(\mathbb R^d)$ if, for any $(x,\mu)\in \mathbb R^d\times\mathcal P_2(\mathbb R^d)$, $g(\cdot,\mu)$ is differentiable at $x$ and $g(x,\cdot)$ is $L$-differentiable at $\mu$. If moreover, $\nabla g(\cdot,\mu)(x)$ and $\partial_\mu g(x,\cdot)(\mu)(y)$ are jointly continuous in $(x,y,\mu)\in\mathbb R^d\times\mathbb R^d\times\mathcal P_2(\mathbb R^d)$, where $\nabla$ is the gradient operator on $\mathbb R^d$, we write $g\in C^{1,(1,0)}(\mathbb R^d\times \mathcal P_2(\mathbb R^d))$. In addition, for $g\in C^{1,(1,0)}(\mathbb R^d\times \mathcal P_2(\mathbb R^d))$,  we write $g\in C_b^{1,(1,0)}(\mathbb R^d\times \mathcal P_2(\mathbb R^d))$ when the  gradient  and the $L$-derivative are uniformly bounded in all variables.

Below, we retrospect the definition on the empirical project of a map; see, for instance, \cite[Definition 5.34]{book:CarmonaDelarue2018a}.

\begin{definition}
Given $u: \cP_2(\bR^d) \to \bR^d$ and $N\in \bN$,  the empirical projection map $u^N$ of $u$ is defined  via the following relation: 
\begin{align*}
    u^N :(\mathbb R^d)^N\ni {\bf x}^N:=(x^1,\ldots, x^N)\mapsto u   \big(\bar{\mu}^N\big)
\quad  \text{with}\quad 
\bar{\mu}^N (\dd y) := \frac{1}{N}\sum\limits_{i=1}^N \delta_{x^i}\, (\dd y). 
\end{align*} 
\end{definition}

Let $u: \cP_2(\bR^d) \to \bR^d$ be a continuously $L$-differentiable map and $u^N$ be the associated empirical projection map.  Then, according to   \cite[Proposition 5.35]{book:CarmonaDelarue2018a}, 
 $u^N$ is differentiable in $(\bR^d)^N$ and satisfies  the identity:
\begin{align}
\label{prop:DerivativeRelations-Space-2-Lions}
    (\nabla_j
    u^N)({\bf x}^N) &= N^{-1}  \partial_\mu u
    (\bar \mu^N)
    (x^j),\quad {\bf x}^N\in(\bR^d)^N,
\end{align}
 where   $\nabla_j u$ denotes the gradient in the $j$-th component $x^j$.


\subsection{McKean-Vlasov SDEs with common noise and associated IPS} 
Throughout the paper,  we assume the following Assumptions.
\begin{enumerate}
\item[$({\bf H}_1)$]there is a constant   $L>0$ such that for   $  t \in[0,T]$, $  x, y\in \bR^d$ and $  \mu, \nu \in \mathcal{P}_2(\bR^d)$, 
\begin{align*}
 |b(t,  x, \mu) - b(t , y, \nu)|\vee\|\sigma^0(t,  x, \mu) - \sigma^0(t , y, \nu)\|_{\rm HS}
 & \vee \|\sigma^1(t,  x, \mu) - \sigma^1(t , y, \nu)\|_{\rm HS}\\
& \leq L\big( |x-y| +W_2(\mu, \nu)\big).\end{align*}

\item[$({\bf H}_2)$]$b(t,\cdot,\cdot),\sigma^0(t,\cdot,\cdot),\sigma^1(t,\cdot,\cdot)\in C_b^{1,(1,0)}(\mathbb R^d)\times \mathcal P_2(\mathbb R^d)$ for all $t\in[0,T]$.
 
\end{enumerate}

It is standard that \eqref{Eq:General MVSDE} is strongly well-posed and has a uniform second-order moment bound in a finite-time interval; see, for example, \cite[Proposition 2.8] {book:CarmonaDelarue2018b} and \cite[Proposition 2.2]{biswas2024radomization} for more details. 
\begin{theorem} 
 \label{existenceuniqueness}
    Under  $({\bf H}_1)$, there exists a unique solution to \eqref{Eq:General MVSDE} in $ L^2(\mathcal C_T;\mathbb P)$ for a finite horizon $T>0$, and   there is a constant $K=K(L,T)>0$ 
    such that 
    $ \bE \|X\|^2_{\mathcal C_T}\leq K\big(1+\bE|\xi|^2\big) $   as long as the initial value $X_0=\xi\in L^2(\Omega,\mathcal F_0,\mathbb P;\mathbb R^d)$. 
\end{theorem}  

In the sequel, 
we copy the dynamics of $(X_t)_{t\ge0}$ solving \eqref{Eq:General MVSDE} to form a \textit{non-interacting particle system (non-IPS for short)}  written as  below: for any $ i=1,\ldots,N$,  
\begin{align}
\label{nonIPS}
    \dd X_t^i 
    = 
    b(t,X_t^i,\mu_t^i) \dd t 
    + 
    \sigma^{0}(t,X_t^i,\mu_t^i) \dd W_t^{0} 
    + 
    \sigma^{1}(t,X_t^i,\mu_t^i) \dd W^{1,i}_t,\quad t>0;  \quad X^i_0 = \xi^i,
\end{align}
where $(\xi^{i})_{1\le i\le N}$ are $L^2(\Omega,\mathcal F_0,\mathbb P;\mathbb R^d)$ i.i.d.~random variables with the same distribution as $\xi$ (all are independent), $((W_t^{1,i})_{t\ge0})_{1\le i\le N}$  are mutually independent Brownian motion on $(\Omega^1,\cF^1, (\cF^1)_{t\ge0}, \bP^1)$. 
In terms of \cite[Proposition 2.11]{book:CarmonaDelarue2018b}, one has 
$\mu_t^i := \Law(X_t^i(\omega^0,\cdot))=\mu^i_t$ for $\bP^0$-almost any $\omega^0$, and meanwhile $\mu_t^i=\mu_t$ for any $t\ge0$, where $(\mu_t)_{t\ge0}$ is determined by \eqref{Eq:General MVSDE}. 
The dynamics of a corresponding  IPS  to \eqref{nonIPS} is formulated  as follows:  
\begin{equation}\label{IPS}
\begin{split}
    \dd X_t^{i,N} = &b(t,X_t^{i,N},\bar{\mu}_t^N) \dd t + \sigma^{0}(t,X_t^{i,N},\bar{\mu}_t^N) \dd W_t^{0}+ \sigma^{1}(t,X_t^{i,N},\bar{\mu}_t^N) \dd W^{1,i}_t, 
\end{split}
\end{equation}
for any $ i=1,\ldots,N$, where  $X^{i,N}_0 := \xi^i$, and the empirical measure $\bar{\mu}_t^N:=\frac{1}{N}\sum_{j=1}^N\delta_{X_t^{j,N}}$. 

The following Lemma plays a vital role in analysing the Malliavin differentiability for McKean-Vlasov SDEs with common noise. 
\begin{lemma} 
\label{wassemp}
Under   $({\bf H}_1)$, 
the IPS \eqref{IPS} has a unique solution $((X_t^{i,N})_{0\le t\le T })_{1\le i\le N}$ and there exists a constant $K=K(T,L)>0$  such that
 $\max_{1\leq l\leq N} \bE  \|X^l\|^2_{\mathcal C_T}\leq K(1+ \bE|\xi|^2) $   in case of $\xi\in L^2(\Omega,\mathcal F_0,\mathbb P;\mathbb R^d)$.
Moreover, for $((X^{i,N}_t)_{0\le t\le N})_{1\le i  \le N}$ the solution to \eqref{nonIPS},  
 \begin{align}
     \label{conditionalpoc}
   \lim_{N\to\infty}\bigg( \max_{1\leq i\leq N}  \bE  \big\|X^{i,N}_\cdot-X_\cdot^i\big\|^2_{\mathcal C_T}+\sup_{0\le t\le T}\bE\mathcal W_2(\bar\mu^N_t,\mu_t^1)^2\bigg)= 0. 
    \end{align}  
\end{lemma}
\begin{proof}
Note that \eqref{IPS} can be rewritten in the state space
$(\mathbb R^d)^N$. Correspondingly, 
under $({\bf H}_1)$, the lift-up drift term and the diffusion term  satisfy the global Lipschitz condition; see e.g. \cite[Lemma 4.4]{reis2023malliavin}. Therefore, it is more or less standard that \eqref{IPS}
is strongly well-posed and the uniform second-order moment estimate is available.  
The assertion  \eqref{conditionalpoc} has been established in   \cite[Theorem 2.12]{book:CarmonaDelarue2018b} so we herein omit the details.  
\end{proof}

\begin{remark}
  Concerning  the validity of 
\cite[Proposition 2.5] {biswas2024radomization}, the higher-order (i.e., $p>4$) moment of the initial distribution is necessitated. Nevertheless, for our purpose, we merely require a qualitative (rather than quantitative) estimate \eqref{conditionalpoc} so the requirement on the higher-order moment on the initial distribution is necessary.   
\end{remark}

\section{Malliavin differentiability}
In this section, we aim at showing that the Malliavin differentiability of $(X_t)_{t\in[0,T]}$ with respect to the common noise and the idiosyncratic noise, one by one.

\begin{theorem}
\label{MD0}
Assume that  $({\bf H}_1)$ and $({\bf H}_2)$ hold. 
 Then, the Malliavin derivative of $(X_t)_{t\ge0}$ solving \eqref{Eq:General MVSDE}  with respect to the common Brownian motion $W^0$, written as $D^0X$,  satisfies the SDE: for all $0\le s \leq t\le T$ 
  \begin{equation}
 \label{commonnoisemalliavin}
    \begin{split}
            D_s^0X_t  &= \sigma^0
        (s,X_s,\mu_s)\\
        &\quad+ \int_s^t\big( \nabla b(r,X_r,\mu_r)D_s^0X_r
        \bE^1\big[\partial_\mu b (r,z,\mu_r)(X_r)D_s^0X_r
        \big ]\big|_{z=X_r}\big) \dd r 
        \\
        &
        \quad+ \int_s^t \big(\nabla \sigma^0(r,X_r,\mu_r)D_s^0X_r
        \bE^1
        \big[\partial_\mu \sigma^0(r,z,\mu_r)
        (X_r)D_s^0X_r\big]\big|_{z=X_r} \big)\dd W_r^{0}  
        \\
        &
        \quad+  \int_s^t\big(\nabla \sigma^1 (r,X_r,\mu_r)D_s^0X_r 
        +\bE^1\big[\partial_\mu \sigma^1  (r,z,\mu_r)    (X_r)D_s^0X_r\big]\big|_{z=X_r}\big)\dd W_r^{1}.
        \end{split} 
        \end{equation} 
 where 
 $D^0_s X_t = {\bf0}_{d\times m_0}$ for $s>t$. 
 Moreover, there exists a unique solution to \eqref{commonnoisemalliavin} and  a positive constant $K=K(T,L)$ such that for any 
\begin{align}
\label{eq:General2ndMomentEstimatefor3.1}
 \sup_{0\le s\le T}
\bE\big[ \big\|D^0_s X_\cdot\big\|^2_{\mathcal C_T}\big] \leq K(1+\bE[|\xi|^2]).
\end{align} 
\end{theorem}

To proceed, we make two comments on Theorem \ref{MD0}.
\begin{remark}
Obviously, Theorem \ref{MD0}, SDE \eqref{commonnoisemalliavin} and \eqref{eq:General2ndMomentEstimatefor3.1} can be written just the same for the non-IPS $(X^i)$ of \eqref{nonIPS} leading to $(D^0X^i)_i$ and \eqref{PP}: namely, for $i=1,\ldots,N$, 
    \begin{equation}
    \label{PP}
    \begin{split}
        D_s^0X_t^i  &= \sigma^0
        (s,X_s^i,\mu_s^i)\\
        &\quad+ \int_s^t\big( \nabla b(r,X_r^i,\mu_r^i)D_s^0X_r^i   
        +  
        \bE^1\big[\partial_\mu b (r,z,\mu_r^i)(X_r^i)D_s^0X_r^i 
        \big ]\big|_{z=X_r^i}\big) \dd r 
        \\
        &
        \quad+ \int_s^t \big(\nabla \sigma^0(r,X_r^i,\mu_r^i)D_s^0X_r^i  
        + 
        \bE^1
        \big[\partial_\mu \sigma^0(r,z,\mu_r^i)  
        (X_r^i)D_s^0X_r^i \big]\big|_{z=X_r^i} \big)\dd W_r^{0}  
        \\
        &
        \quad+  \int_s^t\big(\nabla \sigma^1 (r,X_r^i,\mu_r^i)D_s^0X_r^i  
        +\bE^1\big[\partial_\mu \sigma^1  (r,z,\mu_r^i) 
        (X_r^i)D_s^0X_r^i\big]\big|_{z=X_r^i}\big)\dd W_r^{1,i }.
        \end{split} 
        \end{equation}
For each $i$ the process $D^0X^i$ of \eqref{PP} shares the same distribution as $D^0 X$ of \eqref{commonnoisemalliavin}. 

As a second observation, in \eqref{PP} the first stochastic integral is equal to 
\begin{align}
\label{PP-auxiliary-dW0-expanded}
 \sum_{j=1}^{m_0}\int_s^t \big(\nabla \sigma^0_j(r,X_r^i,\mu_r^i)D_s^0X_r^i 
        + 
        \bE^1
        \big[\partial_\mu \sigma^0_j(r,z,\mu_r^i) 
        (X_r^i)D_s^0X_r^i \big]\big|_{z=X_r^i} \big)\dd W_r^{0,j},   
\end{align}
where $W^{0,j}$ denotes the $j$-th component of $W^{0}$, and $\sigma_i^0$  stands  for the $i$-th column vector of $\sigma^0$. The second stochastic integral in \eqref{PP} can be understood in the same manner. 

\end{remark}

\begin{remark}
A close inspection of the proof for Theorem \ref{MD0} reveals that (i) the qualitative  conditional PoC and (ii) the uniform (in particle number $N$) second-order moment boundedness of $D^0 X^{i,N}$ is essential. Provided that the drift $b$
fulfils a monotone condition in the spatial variable and is Lipschitz continuous under the $L^2$-Wasserstein distance in the measure variable, along with the precondition that $\sigma^0,\sigma^1$ are Lipschitz continuous in the spatial variable and the measure variable,  the essentials (i) and (ii) mentioned previously can also be  established. For the monotone setting, there is much  more sophisticated work to be done in order to implement Step $2$ and Step $3$ in the proof of Theorem \ref{MD0}. Nonetheless, in the present work, we  aim to reveal the spirit on the exploration of Malliavin differentiability for McKean-Vlasov SDEs with common noise. So, the extension of Theorem \ref{MD0} to the setting of the locally Lipschitz coefficients is left for the future work. 
\end{remark}

Below, we move on to complete  the proof of Theorem \ref{MD0}, which is separated into four steps. 
 
\begin{proof}
The proof's main argument is that the Malliavin regularity of the IPS \eqref{IPS} carries through to the non-IPS \eqref{nonIPS} in the particle limit $N$. Thus, we present our proof from the point of view of the Malliavin differentiability of \eqref{nonIPS} and 
\eqref{PP} and not \eqref{Eq:General MVSDE} and \eqref{commonnoisemalliavin}; the result is, of course, the same.

\textit{Step 0. Basic property and well-posedness of the SDE \eqref{PP} and \eqref{commonnoisemalliavin}}.  
Due to (${\bf H}_2$), 
there exists   $C>0$
such that for all  $A\in\mathbb R^{d\otimes m_0}$, $t\in[0,T]$ and $(x,y,\mu)\in\mathbb R^d\times\mathbb R^d\times\mathcal P_2(\mathbb R^d)$, 
\begin{align*}
 \|g(t,x,\mu)A\|_{\rm HS} \le C\|A\|_{\rm HS}\quad \mbox{ and }\quad \|h(t,x,\mu,y)A\|_{\rm HS} \le C\|A\|_{\rm HS},
\end{align*}
where   $g =\nabla b ,\nabla\sigma_i^0 ,\nabla\sigma_i^1 $ and $h(t,x,\mu,y)=\partial_\mu b(t,x,\mu)(y),\partial_\mu\sigma_i^0(t,x,\mu)(y),\partial_\mu\sigma_i^1(t,x,\mu)(y)$. 
With the previous estimates at hand, \eqref{eq:General2ndMomentEstimatefor3.1}
can be obtained from the Burkholder-Davis-Gundy inequality and  the Gronwall inequality in case   \eqref{PP}  is strongly well-posed.

Below, we merely show that \eqref{PP} is well-posed for the case $s=0$ since the general case is analogous. 
To show the strong well-posedness of \eqref{PP} with $s=0$, we appeal to the Banach fixed point approach (similarly to  
\cite[Theorem 1.1]{Ren25}). Define the  set
\begin{align*}
\mathbb H =\Big\{f:[0,T]\times\Omega\rightarrow\mathbb R^{d\otimes m_0}\mbox{ is progressively measurable s.t.}\int_0^T\mathbb E\|f(t)\|_{\rm HS}^2\,\dd t<\infty\Big\},   
\end{align*}
which is a   complete metric space under the metric:
$ 
\rho_\lambda(f,g):=\big(\int_0^T{\rm e}^{-\lambda t} \mathbb E\|f_t-g_t\|_{\rm HS}^2\dd t\big)^{\frac{1}{2}}$, $ f,g\in \mathbb H,  
$  for each $\lambda>0.$ 
Now, we define a map $\eta\ni \mathbb H\mapsto \Phi(\eta) $  to be the right hand side of \eqref{PP} with $D^0X^i$ therein being replaced by $\eta.$ Under $({\bf H}_1)$ and $({\bf H}_2)$, we can infer that $\Phi:\mathbb H\rightarrow\mathbb H$ and $\Phi$ is a contractive map when $\lambda>0$ is chosen large enough. Subsequently, the Banach fixed point theorem enables us to conclude well-posedness.

\textit{Step 1. $L^2$-boundedness of $D^0_s X^{i,N}_t$}. 
Starting from the IPS \eqref{IPS}, as a classical SDE, take its $D^0$ Malliavin derivative: $D^0_s X^{i,N}_t$. 
Clearly,  $D^0_s X^{i,N}_s =0$ for $s > t$, while we obtain from \eqref{prop:DerivativeRelations-Space-2-Lions} (see also \cite{platonovdosreis2023ItowentzelLions}) that for $0\leq s\leq t \leq T <\infty$, 
\begin{equation}
\label{malderivimplicit}
\begin{split}
   \dd 
   D^0_s X^{i,N}_t
    &= \bigg(\nabla b(t,X_t^{i,N},\bar\mu_t^N)D^0_s X^{i,N}_t
    +\frac{1}{N}     \sum_{k=1}^N 
                  \partial_\mu b  (t,X_t^{i,N},\bar\mu_t^N)(X_t^{k,N}) D^0_s X^{k,N}_t \bigg) \dd t \\  
    & 
\quad+\sum_{j=1}^{m_0}\bigg(\nabla \sigma_j^0(t,X_t^{i,N},\bar\mu_t^N)D^0_s X^{i,N}_t\\
    &\qquad\qquad~~+\frac{1}{N}
    \sum_{k=1}^N
      \partial_\mu  \sigma^{0}_j(t,X_t^{i,N},\bar\mu_t^N) (X_t^{k,N}) D^0_s X^{k,N}_t\bigg)
    \dd W_t^{0,j}
    \\
   & +\sum_{j=1}^{m}\bigg(\nabla \sigma_j^1(t,X_t^{i,N},\bar\mu_t^N)D^0_s X^{i,N}_t\\
   &\qquad\qquad+\frac{1}{N}
    \sum_{k=1}^N
      \partial_\mu  \sigma^{1}_j(t,X_t^{i,N},\bar\mu_t^N) (X_t^{k,N}) D^0_s X^{k,N}_t\bigg)
    \dd W_t^{1,i,j}
    \\
    &=:A^i_{s,t}\dd t+\sum_{j=1}^{m_0}B_{s,t}^{i,j}\dd W_t^{0,j}+\sum_{j=1}^{m}C_{s,t}^{i,j}\dd W_t^{1,i,j}, 
\end{split}
\end{equation}
where $D^0_s X^{i,N}_s=\sigma^0(s,X_s^{i,N},\bar\mu_s^N)$. For any fixed $N$, $(D^0 X^{i,N})_{s\leq t}$ is well-defined and well posed by the results in \cite{imkeller2018differentiability} and with exploding norms as $N\to \infty$. 
Next, we show   that such norms do not explode. Indeed,  
 applying It\^o's formula yields that 
\begin{align*}
 & \dd  \|D^0_s X^{i,N}_t \|^2_{\rm HS}
 = 2\langle D^0_s X^{i,N}_t , A_{s,t}^i\rangle_{\rm HS} \dd t 
 + \sum_{j=1}^{m_0}\Big\| B_{s,t}^{i,j} \Big\|^2_{\rm HS}\dd t  
 + \sum_{i=1}^{m}\Big\| C_{s,t}^{i,j} \Big\|^2_{\rm HS}\dd t+\dd M_{s,t}^i, 
\end{align*}
where $M_{s,t}^i$ a    martingale term.

Initially fixing $s$, applying the supremum over $t$ and taking said expectation, one implements the Young and Jensen inequalities and makes use of $({\bf H}_2)$  to obtain the existence of some $C_1>0$ (independent of $N$, but dependent on $d,m$ and $m_0$) such that for all $0\le s\le t\le T$, (with a slight abuse of notation as we write $\|\cdot\|_{\cC_t}$ and not $\|\cdot\|_{\cC_{s,t}}$)
 \begin{align}
 \label{gronwallestimate}
\bE \big[ \big\|D^0_s X^{i,N}_\cdot \big\|^2_{\mathcal C_t}\big]
     &
     \leq 
   \bE  \big[\|\sigma^{0}(s,X_s^{i,N},\bar{\mu}_s^N) \|^2_{\rm HS}\big] \notag
     \\
     &
     + C_1\bE 
     \bigg[ 
      \int_s^t 
    \Big( \| D^0_s X^{i,N}_t \|^2_{\rm HS}
     + \frac{1}{N}\sum_{k=1}^{N} \| D^0_s X^{k,N}_t \|^2_{\rm HS} \Big)
     \dd r\bigg],
 \end{align}
 where the $\bE 
 \big[\sup_{0\leq t \leq T} |M_N(s,t)|\big]$ was treated using the Burkholder-Davis-Gundy inequality. By invoking (${\bf H}_1$), along with Lemma \ref{wassemp}, it follows that for some constants $C_2,C_3>0$ (independent of $N$), 
 we have 
 \begin{equation}\label{ER}
 \begin{split}
    \bE  \big[\|\sigma^{0}(s,X_s^{i,N},\bar{\mu}_s^N) \|^2_{\rm HS}\big] 
     &\leq
C_2 
\bigg(  1 +\bE 
[  |X_s^{i,N}|^2 ] + \frac1N \sum_{k=1}^N 
\bE [ |X_s^{k,N}|^2 ] \bigg)
\\
&
\leq C_3 (1+\bE[|\xi|^2]).
\end{split}
\end{equation}  
Averaging over the particle index $i=1,\ldots,N$, we have
\begin{align} \notag 
 \frac{1}{N}\sum_{i=1}^N \bE \big[ \big\| D^0_s X^{i,N}_\cdot \big\|^2_{\mathcal C_t}\big]
    & 
   \leq 
C_3 (1+\bE[|\xi|^2])   
        + 2C_1 \int_s^t 
        \Bigg[\frac{1}{N}\sum_{i=1}^N\bE \| D^0_s X^{i,N}_r \|^2_{\rm HS}\Bigg]\dd r. 
\end{align}   
Subsequently, Gronwall's inequality enables us to derive that 
\begin{align*}
\label{eq:aux mean moment estimate v1} 
\frac{1}{N}\sum_{i=1}^N \bE \big[ \big\| D^0_s X^{i,N}_\cdot \big\|^2_{\mathcal C_T}\big] 
    \leq  
    C_3 (1+\bE[|\xi|^2]) {\rm e}^{2C_1 T}.
\end{align*}  
Injecting the above estimate  
back into  \eqref{gronwallestimate} and exploiting \eqref{ER}, we obtain
\begin{equation*}\begin{split}\label{E1...NDX2} 
 \bE \big[ \big\| D^0_s X^{i,N}_{s,\cdot} \big\|^2_{\mathcal C_t}\big]   
 &\leq 
   C_3 (1+\bE[|\xi|^2])\\
    &+ C_1 
     \bigg[ 
      \int_s^t 
    \Big( \bE \| D^0_s X^{i,N}_r \|^2_{\rm HS}
     + C_3 (1+\bE[|\xi|^2]) {\rm e}^{2C_1 T} \Big)
     \dd r\bigg]. 
\end{split}
\end{equation*} 
Once  more, applying Gronwall's inequality yields that   
\begin{align*}
    \sup_{N \in \bN}\max_{1\leq i \leq N}\sup_{0\leq s \leq T}  
    \bE \big[ \big\| D^0_s X^{i,N}_\cdot \big\|^2_{\mathcal C_T}\big]<\infty. 
\end{align*} 
This, besides  Lemma \ref{wassemp}, 
implies that $X^i \in \bD^{1,2}$ and $D^0_s X^i_t = \lim_{N\to\infty}D^0_s X^{i,N}_t $ in $L^2(\mathcal H)$ weakly by applying  \cite[Lemma 1.2.3 and Proposition 1.5.5]{nualart2006malliavin}.  
Two points need to be remarked. Firstly, we still need to identify the SDE that this limiting object satisfies. Secondly, for any $i=1,\ldots,N$ $\bP$-a.s.~the processes $X^{i}_t,X^{i,N}_t$ and $D^0_s X^{i}_t,D^0_s X^{i,N}_t$ are continuous in time $t$ (given $s$).

\textit{Step 2. Convergence of the measure derivative terms.} 
Fix $i$. We now prove convergence of the measure derivative terms in \eqref{malderivimplicit} to the corresponding terms in \eqref{PP} (and \eqref{PP-auxiliary-dW0-expanded}). Recall that $\mu^i$ is a $\Omega^0$-random variable only. 
Define the auxiliary quantities  
\begin{align*}
\Lambda_{s,t}^{b,i} := 
   \frac{1}{N}\sum_{k=1}^N
   \partial_\mu b(t,X_t^{i,N},\bar{\mu}_t^N )(X^{k,N}_t) D^0_s X^{k,N}_t  
   - {\bE}^{1} \big[\partial_\mu b (t,z,\mu_t^i ) (X^i_t) D^0_s X^i_{t}\big] \big|_{z=X_t^i}
\end{align*}
  and similarly, comparing \eqref{PP-auxiliary-dW0-expanded} to \eqref{malderivimplicit} and accordingly, define $\Lambda^{\sigma^{0}_j,i}_{s,t}$, $\Lambda^{\sigma^{1}_j,i}_{s,t}$.  
   Note the quantities in the difference are well defined from the results established in the previous step.  
  Taking the conditional square expectation,  
\begin{align*}
&
   \sup_{0\leq s\leq T}\bE^{1}\bigg[\sup_{0\leq r \leq T}\bigg\|
   \frac{1}{N}\sum_{k=1}^N
   \partial_\mu b(r,X_r^{i,N},\bar{\mu}_r^N)(X^{k,N}_r) D_s^0 X_r^{k,N} 
   \\
   &
   \hspace{5cm}- {\bE}^{1} \big[ \partial_\mu b(r,z,\mu_r^i)( X^i_r ) D_s^0 X^i _r\big]\big|_{z=X^i_r}
   \bigg\|^2_{\rm HS} \bigg]
   \\
   &
   \leq 4 \sup_{0\leq s\leq T}\bE^{1}\bigg[\sup_{0\leq r \leq T} \bigg\|
   \frac{1}{N}\sum_{k=1}^N
   \partial_\mu b(r,X_r^{i,N},\bar{\mu}_r^N) (X_r^{k,N}) D_s^ 0X_r^{k,N}
   \\
   &
    \hspace{5cm}- \frac{1}{N}\sum_{k=1}^N
   \partial_\mu b(r,X_r^{i,N},\mu_r^i)(X_r^{k,N}) D_s^0 X_r^{k,N} \bigg\|^2_{\rm HS}\bigg] 
   \\
   &
    + 4 \sup_{0\leq s\leq T}\bE^{1}\bigg[\sup_{0\leq r \leq T}\bigg\|\frac{1}{N}\sum_{k=1}^N
   \partial_\mu b(r,X_r^{i,N},\mu_r^i)(X_r^{k,N}) D_s^0 X_r^{k,N}
   \\
   &
    \hspace{6cm}- 
    \frac{1}{N}\sum_{k=1}^N 
    \partial_\mu b(r,X_r^i,\mu_r^i)(X^k_r) D_s^0 X_r^{k,N}
   \bigg\|^2_{\rm HS}\bigg] 
   \\
   &
    + 4 \sup_{0\leq s\leq T}\bE^{1}\bigg[\sup_{0\leq r \leq T}\bigg\|\frac{1}{N}\sum_{k=1}^N
  \partial_\mu b(r,X_r^i,\mu_r^i)(X^k_r) D_s^0 X_r^{k,N}  
  \\ 
  &
   \hspace{6cm}- \frac{1}{N}\sum_{k=1}^N
   \partial_\mu b(r,X_r^i,\mu_r^i) (X^k_r) D_s^0 X_r^k\bigg\|^2_{\rm HS}\bigg] 
   \\
   &
   + 4 \sup_{0\leq s\leq T}\bE^{1}\bigg[\sup_{0\leq r \leq T}\bigg\|\frac{1}{N}\sum_{k=1}^N
   \partial_\mu b(r,X_r^i,\mu_r^i) (X^k_r)D_s^0X_r^k 
   \\
   &
    \hspace{6cm} - {\bE}^{1} \big[ 
    \partial_\mu b(r,z,\mu_r^i)( X^i_r ) D_s^0 X^i _r\big]\big|_{z=X^i_r}
    \bigg\|^2_{\rm HS}\bigg]. 
\end{align*}
\textit{For the first term.} Using Lemma \ref{wassemp}, the joint continuity of $\partial_\mu b$ and Jensen's Inequality we bound $\bP^0$-a.s.~this term by $4\sup_{0\leq s\leq T}\bE^{1}\big[\sup_{0\leq r \leq T}|\frac{1}{N}\sum_{k=1}^N D_s^0X^k_r|^2\big]\cdot \varepsilon_N$ for some bounded sequence $(\varepsilon_N)_{N\geq 1}\to 0$ as $N \to\infty$. The $\bP^0$-a.s.~boundedness of $\sup_{0\leq s\leq T}\bE^{1}\big[\sup_{0\leq r \leq T}|\frac{1}{N}\sum_{k=1}^N D_s^0 X^{k,N}_r|^2\big]$, which is justified as in \textit{Step 1}, implies this converges $\bP^0$-a.s.~to zero in the limit $N \to \infty$.  

\textit{For the second term}: the (conditional) Propagation of Chaos ascertains that $X_r^{k,N}$ converges $\bP^0$-a.s.~to $X_r^k$ in $L^2(\Omega^{1})$ and since, by $({\bf H}_2)$, $\partial_\mu b$ is jointly continuous and uniformly bounded then the convergence results follows immediately from the $L^2(\Omega^{1})$-boundedness of $D_s^0 X_r^{k,N}$ (uniformly in $N$) from \textit{Step 1} and conditional dominated convergence.

\textit{For the third term.} We also have convergence to zero, $\bP^0$-a.s.~in the limit $N\to\infty$. This is justified by boundedness and uniform continuity of $\partial_\mu b$ and the already established fact (in the proof's \textit{Step 1}) that $D^0_s X^i_t = \lim_{N\to\infty} D^0_s X^{i,N}_t$ in $L^2(\cC_T)$-weakly, before once again applying conditional dominated convergence.  

\textit{For the fourth term}. For the particle system $\{(X^k,D^0 X^k)\}_k$, one  observes by strong symmetry that the particles are exchangeable (and pairwise conditionally independent with uniformly (in $N$) finite second moments). That is, 
\begin{align*}
    \{(X_r^{\pi(1)},D_s^0 X_r^{\pi(1)}),\ldots,(X_r^{\pi(N)},D_s^0 X_r^{\pi(N)})\} \stackrel{\text{Law}}= \{(X_r^1,D_s^0X_r^1),\ldots, (X_r^N,D_s^0X_r^N)\}
\end{align*}
for any permutation $\pi$ of $\{1,\ldots,N\}$.  
By the strong Law of Large Numbers, $\bP^0\otimes\bP^{1}$-a.s. 
\begin{align*}
    \lim_{N\to\infty }\bigg|
    \Big\{
    \frac{1}{N}\sum_{k=1}^N
   \partial_\mu b(r,z,\mu_r^i)(X^k_r)D_s^0X_r^k 
   - \bE^{1}\big[\partial_\mu b (r,z,\mu_r^i)(X^i_r) D_s^0 X_r^i\big]\big]\Big\}\Big|_{z=X^i_r}
   \bigg|^2 =0,
\end{align*}
where the case $k=i$ is trivial using the boundedness of $\partial_\mu b$, the $L^2$-bounds of $D^0X^\cdot$ established earlier and that the weight $1/N$ is vanishing; the continuity in time of the processes extends the result uniformly over $[s,T]$; the result holds with ${z=X^i_r}$ omitted and for any $z\in\bR^d$.  
Uniform boundedness of $\partial_\mu b$, and the $L^2(\Omega^{1})$-uniform bounds on $D^0_s X^i_t$ and $D^0_s X^{i,N}_t$, justified in \emph{Step 1}, allow us to apply conditional dominated convergence and conclude this final term converges to zero $\bP^0\otimes\bP^1$-a.s.
Thus for any fixed $i=1,\ldots,N$ we have $\bP^0$-a.s.~ 
 \begin{align*}    
 \limsup_{N\to\infty}\bE^{1}\Big[\sup_{0\leq s\leq r\leq T}|\Lambda^{b,i}_{s,r}|^2\Big] = 0.
 \end{align*} 

Gathering the arguments, an application of the Cauchy-Schwarz inequality then yields the required convergence
\begin{align}
\label{lambdabconvergence}
\int_s^t & 
\frac{1}{N}\sum_{k=1}^N  
   \partial_\mu b
   (r,X_r^{i,N},\bar{\mu}_r^N)(X^{k,N}_r)
       D_s^0 X_r^{k,N} \dd r \notag
   \\
   &
   \longrightarrow \int_s^t {\bE}^1\Big[ 
   \partial_\mu b (r,z,\mu_r^i)(X_r^i) D_s^0 X_r^i \Big]\bigg|_{z=X^i_r}
   \dd r \ \textrm{ as } \ N\to \infty 
   \end{align}
   in $L^2(\Omega^{1})$ for any $t\in[s,T]$ conditional on a path $\omega^0$ $\bP^0$-a.s.  
   To deal with the corresponding stochastic integral terms involving the measure derivatives of $\sigma^{0}_j$, $j=1,\ldots,m_0$ and $\sigma^{1}_j$, $j=1,\ldots,m$, we apply the above method, replacing the Cauchy-Schwarz inequality with an additional It\^o Isometry argument, localising and arguing over quadratic variation.   
   
\textit{Step 3. Identifying the limiting equation of \eqref{malderivimplicit}.}  With the results of \textit{Step 2} at hand, the remaining arguments are classical and thus presented in a streamlined fashion. 
   Fix $i$. Set for $s\in[0,T]$ 
   \begin{align*} 
       \delta\sigma_s^{0,i} := 
       \sigma^{0}\big( s, X_s^{i,N},\bar{\mu}_s^N  \big) 
       - \sigma^{0}\big( s, X_s^i,\mu_s^i \big),  
   \end{align*}
   and compare the McKean-Vlasov SDEs \eqref{malderivimplicit} and \eqref{PP}. We prove next that \eqref{PP} is the $L^2$-limit equation of \eqref{malderivimplicit} (for any fixed $i$ as $N\to \infty$ ):  
   \begin{align*}
       D_s^0 X_t^{i,N} - D_s^0 X_t^i 
       =& \delta\sigma_s^{0,i}
       +\int_s^t \widehat A_{s,r}^i \hspace{0.1cm}\dd r + \int_s^t \Lambda^{b,i}_{s,r} 
        \hspace{0.1cm}\dd r
        \\
        &
        +\sum_{j=1}^{m_0}\int_s^t \big( \widehat B_{s,r}^{i,j}   +  \Lambda^{\sigma^0_j,i}_{s,r} \big) 
        \hspace{0.1cm}\dd W_r^{0,j}
        +\sum_{j=1}^{m}\int_s^t \big( \widehat C_{s,r}^{i,j}   +  \Lambda^{\sigma^1_j,i}_{s,r} \big)
        \hspace{0.1cm}\dd W^{1,i,j}_r,
   \end{align*}
where the $\Lambda$ processes were defined in Step 2 and we define additionally
\begin{align*}
&\widehat A_{s,r}^i := \nabla b (r,X_r^{i,N},\bar{\mu}_r^N) D_s^0 X_r^{i,N} - \nabla b (r,X_r^i,\mu_r^i) D_s^0 X^i_r,
\\
&
\widehat B_{s,r}^{i,j} :=\nabla \sigma^{0}_j (r,X_r^{i,N},\bar{\mu}_r^N ) D_s^0 X_r^{i,N} 
       - \nabla \sigma^{0}_j (r,X_r^i,\mu_r^i) D_s^0 X_r^i , \qquad j=1,\ldots,m_0
\\
&
\widehat C_{s,r}^{i,j} := \nabla \sigma^{1}_j (r,X_r^{i,N},\bar{\mu}_r^N ) D_s^0 X_r^{i,N} 
      - \nabla \sigma^{1}_j  (r,X_r^i,\mu_r^i ) D_s^0 X_r^i, \qquad j=1,\ldots,m. 
\end{align*}

Taking the square expectation, and using the It\^o isometry, we obtain that
\begin{equation}
    \label{PsiN eq}
    \begin{split}
    \bE\Big[\|D_s^0 X_t^{i,N} - D_s^0 X_t^i\|^2_{\rm HS}\Big]
    -\Psi_N  
   & =  \bE\left[\left\|\int_s^t \widehat A^i_{s,r} \dd r\right\|^2_{\rm HS}\right]
    +\bE\left[\sum_{j=1}^{m_0} \int_s^t \|\widehat B_{s,r}^{i,j}\|^2_{\rm HS}\dd r\right] \\  
   &\quad +\bE\left[ \sum_{j=1}^{m}\int_s^t \|\widehat C_{s,r}^{i,j}\|^2_{\rm HS} \dd r\right], 
\end{split}
\end{equation}
where 
\begin{align*}
    \Psi_N:= 
    &
    \bE\left[  \|\delta\sigma^{0,i}_s\|^2_{\rm HS} \right] 
    + 2\bE\left[\int_s^t\langle\delta\sigma^{0,i}_s, \widehat A_{s,r}^i\rangle_{\rm HS} \dd r\right]
    +2 \bE\left[\int_s^t\langle\delta\sigma^{0,i}_s, \Lambda^{b,i}_{s,r}\rangle_{\rm HS}\dd r\right] 
    \\
    &
    + 2\bE\left[\Big\langle\int_s^t \widehat A^i_{s,r} \dd r ,\int_s^t \Lambda^{b,i}_{s,r}\dd r\Big\rangle_{\rm HS}\right]
    +2\bE\left[\sum_{j=1}^{m_0}\int_s^t \langle \widehat B^{i,j}_{s,r}, \Lambda^{\sigma_j^0,i}_{s,r}\rangle_{\rm HS}\dd r\right]
    \\
    &
    +2\bE\left[\sum_{j=1}^{m}\int_s^t \langle \widehat C^{i,j}_{s,r}, \Lambda^{\sigma^1_j,i}_{s,r} \rangle_{\rm HS}\dd r\right]
    +\bE\left[\left\|\int_s^t \Lambda^{b,i}_{s,r} \dd r\right\|^2_{\rm HS}\right] 
    \\
    &
    +\bE \left[\sum_{j=1}^{m_0}\int_s^t\|\Lambda^{\sigma^0_j,i}_{s,r}\|^2_{\rm HS} \dd r\right]   
    + \bE\left[\sum_{j=1}^m\int_s^t \|\Lambda^{\sigma^1_j,i}_{s,r}\|^2_{\rm HS} \dd r\right]\color{black},
\end{align*}
 noting that since $W^0$ and $W^{1,\cdot}$ are independent no cross variation terms appear. 
 
The joint continuity property of $\sigma^{0}$ and the (conditional) Propagation of Chaos implies $L^2(\cC_T)$ and $\bP^0$-a.s.~convergence of $\delta\sigma^{0,i}_\cdot$ and  $\bE^1 [\, \|\delta\sigma^{0,i}_\cdot\|^2_{\rm HS}]\to 0$ respectively as $N\to \infty$. 
The last three terms of $\Psi_N$ vanish using the results of \textit{Step 2} (possibly along a subsequence if needed). 
The remaining five terms also vanish in the $N$ limit using classic arguments: one mimics as needed the arguments used in \textit{Step 2}, using the tower property $\bE[\cdot]=\bE[\bE^1[\cdot]]$ to draw on the conditional Propagation of chaos \eqref{conditionalpoc}, using the joint continuity and uniform boundedness of the derivatives $\nabla b$, $\nabla \sigma^{0}$ and $\nabla \sigma^{1}$ plus the $L^2$-convergence (possibly along a subsequence) and $L^2$-bounds of $D^0X^{i,N},D^0X^i$ coupled with repeated use of the Cauchy-Schwarz and Jensen inequalities and the convergence of the $\Lambda$ terms established in \textit{Step 2}.  
In conclusion, the LHS of \eqref{PsiN eq} converges to zero (possibly along a subsequence if needed), thus so must its RHS (using additionally the continuity in time of the processes involved) and it follows that \eqref{PP} is precisely the $L^2$-limiting equation of the IPS \eqref{malderivimplicit}. 
\end{proof}
The Malliavin differentiability of \eqref{Eq:General MVSDE} with respect to the idiosyncratic $W^1$ is classical. 
\begin{theorem}
\label{MD1}
Assume $({\bf H}_1)$ and let $b,\sigma^0,\sigma^1$ be continuously differentiable in space with uniformly bounded spatial derivatives. Then, the Malliavin derivative of \eqref{Eq:General MVSDE} with respect to the idiosyncratic Brownian motion $(W^1_t)_{t\ge0}$ satisfies the linear SDE: for any $t\ge s$,
    \begin{equation*} 
    \begin{split}
        D_s^1 X_t =  
        &
        \sigma^1(s,X_{s}, \mu_s)
        + \int_s^t \nabla b(r,X_r,\mu_r) D_s^1 X_r \dd r
        + \int_s^t \nabla \sigma^0(r,X_r,\mu_r) D_s^1 X_r \dd W^{0}_r   
        \\
        &
        + \int_s^t\nabla \sigma^1(r,X_r,\mu_r) D_s^1 X_r \dd W^{1}_r,
        \end{split}
        \end{equation*}
        and  $D^1_s X_t = {\bf0}_{d\times m}$ for any $s>t$. Moreover, there is a constant   $
        K=K(T,L)>0$  such that 
\begin{align*}
\sup_{0\leq s\leq T} \bE\big[\big\|D^1_s X_{.}\big\|_{\mathcal C_T}^2 \big]
\leq K(1+\bE|\xi|^2).
\end{align*} 
\end{theorem}
We point that using further approximation arguments, this result holds only under $({\bf H}_1)$ (understanding $\nabla b,\nabla \sigma^0,\nabla \sigma^1$ in a generalised sense, see \cite[Prop.~1.2.4]{nualart2006malliavin}).
\begin{proof}
The Malliavin derivative with respect to the idiosyncratic noise $(W^1_t)_{t\ge0}$ of $(X_t)_{t\ge0}$ \eqref{Eq:General MVSDE} will ignore the conditional dependence on $\omega^0 \in \Omega^0$, in the sense that the Cameron-Martin perturbations are solely against the Brownian motion that is not being conditioned upon.  
Thus, the result follows promptly from the same arguments as in  \cite[Theorem 3.4]{reis2023malliavin} and \cite[Proposition 3.1]{Crisan2018}. 
Methodologically: via a Picard iteration argument build a convenient sequence of \textit{standard SDEs} that converges to the McKean-Vlasov SDE, apply a suitable Malliavin differentiability result for regular SDEs \cite{imkeller2018differentiability} to the sequence and then pass to limit using the closedness of the Malliavin operator  \cite[Lemma 1.2.3]{nualart2006malliavin}.
\end{proof}


\begin{section}{Conditional integration by parts formula} 
For a random variable  $\eta,$ we write $[\eta]$ as its law. 
For convenience, denote the solution to \eqref{Eq:General MVSDE} as $(X_t^{[\xi]})_{t\ge0}$ to highlight the distribution of $\xi$,  and, for each $t\ge0,$ write $\mu_t^\xi := \Law(X_t^{[\xi]}(\omega^0,\cdot))$.  
Consider  the decoupled SDE associated with \eqref{Eq:General MVSDE}: 
\begin{align}
\label{decoupledsde}
    &\dd X_t^{x,[\xi]} = b\big(t,X_t^{x,[\xi]},\mu_t^\xi\big)\dd t + 
    \sigma^0\big(t,X_t^{x,[\xi]},\mu_t^\xi\big) \dd W^{0}_t+ 
    \sigma^1\big(t,X_t^{x,[\xi]},\mu_t^\xi\big) \dd W^{1}_t,
    \quad  t>0,
\end{align}
with the initial value $X_0^{x,[\xi]}=x\in\bR^d$. 
The SDE above is a standard one with random coefficients and is related but decoupled from \eqref{Eq:General MVSDE}   
     since the measure term is exogenous (and also independent of $x$). Its well-posedness follows from \cite[Theorem 3.1.1]{prevotrockner2007Book}, observing the integrability of the maps $t\mapsto b(t,0,\mu_t^\xi)$ and $t\mapsto \sigma^l(t,0,\mu_t^\xi)$, $l=0,1$.  
     Results regarding the first variation process with respect to the initial condition follow the almost identical formulation as the unconditional McKean-Vlasov SDE, as presented in \cite{Crisan2018}. We state these for completeness.
    \begin{lemma} 
Let $({\bf H}_1)$ and $({\bf H}_2)$ hold. 
Then,  the map $ \bR^d \ni x \mapsto X_\cdot^{x,[\xi]} \in L^2(\mathcal C_T,\mathbb P) $ is $\bP$-a.s. continuously differentiable and $w_t^x :=(\nabla X_t^{\cdot,[\xi]}) (x)$ satisfies the affine SDE: for any   $0\le s\le t\le T$,
  \begin{align}
  \label{xvariation}
        w_t^x = I_{d\times d}&
        + \int_s^t 
        \nabla b(r,\cdot,\mu_r^\xi)(X_r^{x,[\xi]})
        w_r^x \dd r +\int_s^t \nabla \sigma^0(r,\cdot,\mu_r^\xi)(X_r^{x,[\xi]})w_r^x \dd W^{0}_r 
        \\ \notag
        & 
        + \int_s^t \nabla \sigma^1(r,\cdot,\mu_r^\xi)(X_r^{x,[\xi]})w_r^x \dd W^{1}_r.
        \end{align}
\end{lemma}
\begin{proof}
Given $X^{x,[\xi]}$ and $\mu^\xi$, and knowing $\nabla b,\nabla \sigma^0,\nabla \sigma^1$ are uniformly bounded, then the linear  SDE \eqref{xvariation} is   well-posed; see e.g. \cite{prevotrockner2007Book,imkeller2018differentiability}. 
For the differentiability statement, recall that the measure component in \eqref{decoupledsde} is independent of the spatial variable $x$ and is exogenous given with good integrability properties (Theorem  \ref{existenceuniqueness}). 
The differentiability of $x\mapsto X^{x,[\xi]}$ follows as a particular case of the results from \cite{imkeller2018differentiability}.  
Refer to \cite[Theorem 5.29]{book:CarmonaDelarue2018b},  which discusses differentiability in a decoupled McKean-Vlasov FBSDE case. 
 \end{proof}

Comparing the SDE in Theorem \ref{MD1} with the SDE  \eqref{xvariation}, we note that these are the same type affine SDEs but with different dimensions and 
initial conditions. Working under the (uniform ellipticity) assumption that there exists some $\delta>0$ such that for all $\zeta\in\bR^d$ and $(t,x,\mu)\in [0,T] \times\bR^d \times \cP_2(\bR^d)$, $\zeta^\top((\sigma^1(\sigma^1)^\top)(t,x,\mu))\zeta \geq \delta |\zeta|^2$, one establishes, exactly as in \cite[Proposition 4.1]{Crisan2018}, the following integration by parts formula. Let $f \in C_b^\infty(\bR^d)$ and $\Phi:[0,T]\times\bR^d\times\cP_2(\bR^d)\to \bD^{1,2}$ 
be a map whose measure derivative and spatial derivative maps are $L^1$-integrable and of linear growth. Then it holds that
 \begin{align*} 
    \bE^1\Big[\nabla f(X_t^{\cdot,[\xi]})(x)
     \Phi(t,x,[\xi])\Big]
    = \frac1t\bE^1\Big[f(X_t^{x,[\xi]})\delta^1\big(r \mapsto g(r)^\top\Phi(t,x,[\xi])\big)\Big],
\end{align*}
where $g(r) :=( (\sigma^1)^\top (\sigma^1(\sigma^1)^\top)^{-1})(r,X_r^{x,[\xi]},\mu_r^\xi)(\nabla X_r^{\cdot,[\xi]})(x)$ and $\delta^1$ is the Skorokhod operator with respect to the idiosyncratic Brownian motion.

More of interest  is an \textit{integration by parts formula for Equation \eqref{Eq:General MVSDE} with respect to the measure variable}. We first need the following lemma. 
\begin{lemma}
\label{lemma42}
    Under $({\bf H}_1)$ and $({\bf H}_2)$, 
    there exists a modification of $X_t^{[\xi]}$ such that the  map $L^2(\Omega) \ni \xi \mapsto X_\cdot^{\xi} \in L^2(\mathcal C_T,\mathbb P)$ is $\bP$-a.s.~Fréchet differentiable in $L^2(\Omega)$. Hence $\partial_\mu X_t^{[\xi]} $ exists and for $v \in \bR^d$, $\Gamma_t^{\xi}(v):= \partial_\mu X_t^{[\xi]}(v)$ satisfies the dynamics:  
    \begin{align}   
        &\Gamma_t^{\xi}(v)
        = \mathbf{\Psi}_t(X)
        +
             \int_0^t\Big( \nabla b(r,\cdot,\mu_r^\xi)(X_r^{[\xi]})\Gamma_r^{\xi}(v) +\bE^1\Big[\partial_\mu b (r,\eta)({X}_r^{[\xi]})\Gamma_r^{\xi}(v) \Big]\Big|_{\eta=(X_r^{[\xi]},\mu_r^\xi)}\Big)
             \dd r\notag
        \\
        &
        + \int_0^t \Big( \nabla \sigma^0(r,\cdot,\mu_r^\xi)(X_r^{[\xi]})\Gamma_r^\xi(v)
        +   \bE^1\Big[ \partial_\mu \sigma^0 (r,\eta)({X}_r^{[\xi]})\Gamma_r^{\xi}(v) \Big]\Big|_{\eta=(X_r^{[\xi]},\mu_r^\xi)}\Big) \dd W_r^{0} \notag
        \\ 
        &
        + \int_0^t \Big( \nabla  \sigma^1(r,\cdot,\mu_r^\xi)(X_r^{[\xi]})\Gamma_r^\xi(v) 
        + \bE^1\Big[ \partial_\mu \sigma^1 (r,\eta)({X}_r^{[\xi]})\Gamma_r^{\xi}(v)\Big]\Big|_{\eta=(X_r^{[\xi]},\mu_r^\xi)}\Big)\dd W_r^{1},
     \label{muvariation}
\end{align}
where setting $\Xi_r:=(r,X_r^{[\xi]},\mu_r^\xi)$ for $r\geq 0$ we have   
\begin{align*}
\mathbf{\Psi}_t(X) :
&= \int_0^t \partial_\mu b (\Xi_r)(X_r^{v,[\xi]})\nabla_x X_r^{v,[\xi]} \dd r 
+ \int_0^t  \partial_\mu \sigma^0(\Xi_r)(X_r^{v,[\xi]})\nabla_xX_r^{v,[\xi]} \dd W_r^{0}
\\
&\quad
+
\int_0^t \partial_\mu \sigma^1(\Xi_r)(X_r^{v,[\xi]})\nabla_xX_r^{v,[\xi]} \dd W_r^{1}.  \notag
\end{align*} 
        
\end{lemma}
\begin{proof}
    See \cite[Lemma 3.5]{MR4278420} and replace their $\Lambda_t^\mu$ by $\Tilde{\Lambda}_t^\xi := [\xi]\circ (\bE\big[X_t^{x,[\xi]}|\cF^0_t \big])^{-1}$, yielding their ``\textit{image-dependent SDE}''. See also \cite[Lemma 5.27]{book:CarmonaDelarue2018b}.   
\end{proof}

We now state the below integration by parts formula that is not covered by \cite{MR4278420}. 
\begin{lemma} 
  Assume   that $({\bf H}_1)$ and $({\bf H}_2)$ hold. Suppose further that 
    there exists some $\delta>0$ such that $\zeta^\top((\sigma^0(\sigma^0)^\top)(t,x,\mu))\zeta \geq \delta |\zeta|^2$ for all $\zeta\in\bR^d$ and $(t,x,\mu)\in [0,T] \times\bR^d \times \cP_2(\bR^d)$. Then,  there exists an $h\in \mathcal H$ (given explicitly in the proof) such that the following integration by parts formula holds: for $v \in \bR^d$ and $f \in C_b^2(\bR^d)$, 
    \begin{align*}
        \bE^1
    \Big[ \partial_\mu (
    f(X_t^{[\xi]}))(v)\Big]
    =
    \frac1t\bE^1\Big[f(X_t^{[\xi]})\delta^0\big(r \mapsto h(r,v)^\top\big)\
    \Big].
    \end{align*}
     where $\delta^0$ denotes the Skorokhod operator with respect to the common Brownian motion; the $L^2$-adjoint to the Malliavin divergence.
\end{lemma}
\begin{proof}
    Compare \eqref{PP} and \eqref{muvariation}: these equations are driven by the same dynamics but with different dimensions and initial conditions. Arguing as in  \cite{Crisan2018}, we have,  with $\mathbf{\Psi}_t(X)$ as above, the $\bP $-a.s.~representation\footnote{ From \eqref{muvariation}, it is easy to see that $\Gamma_\cdot^\xi(v)$ can be written with the help of $\Psi_\cdot(X)$
\begin{align*}
\Gamma_t^\xi(v)=\Gamma_r^\xi(v)+\Psi_t(X) - \Psi_r(X) +\int_r^t \cdots \textrm{``extra terms''}\cdots .
\end{align*}
For the relationship between $\partial_\mu X_t^\xi(v)$ and $\Gamma_t^\xi(v)$, we formally multiply $( (\sigma^0(\sigma^0)^\top)^{-1})(r)\big(\Gamma_r^\xi(v)+\Psi_t(X) - \Psi_r(X)\big)$ on both sides of \eqref{PP} but as the term is not adapted it cannot be moved inside the stochastic integrals. Thus, methodologically, the relations between $\partial_\mu X_t^\xi(v)$ and $\Gamma_t^\xi(v)$ does not follow via  strong uniqueness. 
}
 \begin{align*}
     \partial_\mu X_t^{[\xi]}(v) 
     = 
     D_r^0 X_t^{[\xi]}  h(r,v),
 \end{align*}
 where
 \begin{align*}
   h(r,v): = \big((\sigma^0)^\top (\sigma^0(\sigma^0)^\top)^{-1}\big)(r,X_r^{[\xi]},\mu_r^\xi)\big\{\mathbf{\Psi}_t(X)- \mathbf{\Psi}_r(X)+\partial_\mu X_r^{[\xi]}(v)\big\}.  
 \end{align*} 
Uniform boundedness of $\partial_\mu \big( f(X_t^{[\xi]})\big)$ follows from \cite[Theo.~5.29]{book:CarmonaDelarue2018b}. Thus,  Malliavin integration by parts implies that for $v \in \bR^d$
\begin{align*} 
    \bE^1
    \Big[\partial_\mu \big( f(X_t^{[\xi]})\big)(v)\Big] 
    &=\frac{1}{t}\bE^1 \Big[\int_0^tD_r^0 X_t^{[\xi]} h(r)(\nabla f)(X_t^{[\xi]})\dd r\Big]\notag
    \\
    &
    = \frac{1}{t}\bE^1 \Big[\int_0^t  h(r)^\top D_r^0f(X_t^{[\xi]})\dd r\Big]
    = \frac{1}{t}\bE^1\Big[f(X_t^{[\xi]})\delta^0\big(r \mapsto h(r,v)^\top\big)\Big].
\end{align*}
\end{proof}


\end{section}


\begin{thebibliography}{10}
\bibitem{BRW}
J.~Bao,  P.~ Ren and F.-Y.~Wang. Bismut formula for Lions derivative of distribution-path dependent SDEs. \newblock{J. Differential Equations}, 282:285--329, 2021



\bibitem{biswas2024radomization}
S.~Biswas, C.~Kumar, Neelima, G.~dos Reis, and C.~Reisinger.
\newblock An explicit {M}ilstein-type scheme for interacting particle systems and {M}c{K}ean-{V}lasov {SDE}s with common noise and non-differentiable drift coefficients.
\newblock {\em Ann. Appl. Probab.}, 34(2):2326--2363, 2024.

\bibitem{bugini2024malliavin}
F.~Bugini, M.~Coghi, and T.~Nilssen.
\newblock Malliavin calculus for rough stochastic differential equations.
\newblock {\em arXiv preprint arXiv:2402.12056}, 2024.

\bibitem{book:CarmonaDelarue2018a}
R.~Carmona and F.~Delarue.
\newblock {\em Probabilistic Theory of Mean Field Games with Applications I}, volume~83 of {\em Probability Theory and Stochastic Modelling}.
\newblock Springer-Verlag, 2018.

\bibitem{book:CarmonaDelarue2018b}
R.~Carmona and F.~Delarue.
\newblock {\em Probabilistic Theory of Mean Field Games with Applications II}, volume~84 of {\em Probability Theory and Stochastic Modelling}.
\newblock Springer-Verlag, 2018.

\bibitem{Crisan2018}
D.~Crisan and E.~McMurray.
\newblock Smoothing properties of {M}c{K}ean-{V}lasov {SDE}s.
\newblock {\em Probab. Theory Related Fields}, 171(1-2):97--148, 2018.

\bibitem{platonovdosreis2023ItowentzelLions}
G.~dos Reis and V.~Platonov.
\newblock It\^{o}-{W}entzell-{L}ions {F}ormula for {M}easure {D}ependent {R}andom {F}ields under {F}ull and {C}onditional {M}easure {F}lows.
\newblock {\em Potential Anal.}, 59(3):1313--1344, 2023.

\bibitem{reis2023malliavin}
G.~dos Reis and Z.~Wilde.
\newblock Malliavin differentiability of {McK}ean-{V}lasov {SDE}s with locally {L}ipschitz coefficients.
\newblock {\em arXiv preprint arXiv:2310.13400}, 2023.

\bibitem{imkeller2018differentiability}
P.~Imkeller, G.~dos Reis, and W.~Salkeld.
\newblock Differentiability of {SDE}s with drifts of super-linear growth.
\newblock {\em Electron. J. Probab.}, 24:Paper No. 3, 43, 2019.


\bibitem{LSZ}D.~Lacker,    M~Shkolnikov    and  J.~Zhang. 
Superposition and mimicking theorems for conditional McKean-Vlasov equations,
{\it J. Eur. Math. Soc. (JEMS)},   25 (8):3229--3288, 2023.



\bibitem{nualart2006malliavin}
D.~Nualart.
\newblock {\em The {M}alliavin calculus and related topics}, volume 1995.
\newblock Springer, 2006.

\bibitem{prevotrockner2007Book}
C.~Pr\'ev\^ot and M.~R\"ockner.
\newblock {\em A concise course on stochastic partial differential equations}, volume 1905 of {\em Lecture Notes in Mathematics}.
\newblock Springer, Berlin, 2007.

\bibitem{Ren25}
P.~Ren.  Extrinsic derivative formula for distribution dependent SDEs. \newblock{\em Bernoulli}, 31 (3),  2508--2524, 2025.



\bibitem{RW}P.~Ren and F.-Y.~Wang. 
 Bismut formula for Lions derivative of distribution dependent SDEs and applications. \newblock{\em J. Differential Equations},  267(8),   4745--4777, 2019



\bibitem{MR4278420}
F.-Y. Wang.
\newblock Image-dependent conditional {M}c{K}ean-{V}lasov {SDE}s for measure-valued diffusion processes.
\newblock {\em J. Evol. Equ.}, 21(2):2009--2045, 2021.

\bibitem{W2013}
 F.-Y. Wang. \newblock Harnack inequalities for stochastic partial differential equations.   Springer, New York, 2013. 



\end{thebibliography}

\end{document}